\documentclass[11pt]{article}

\usepackage{amssymb,amsmath,amsthm}
\usepackage{geometry}
\usepackage{color}

\def\Ex{{\mathbb E}}
\def\Pr{{\mathbb P}}
\def\er{{\mathbb R}}
\def\ind{{\mathbf 1}}

\newtheorem{lem}{Lemma}
\newtheorem{conj}[lem]{Conjecture}
\newtheorem{thm}[lem]{Theorem}
\newtheorem{prop}[lem]{Proposition}

\theoremstyle{remark}
\newtheorem{rem}[lem]{Remark}

\theoremstyle{definition}

\title{Weak and strong moments of $\ell_r$-norms of log-concave vectors
\thanks{Research supported by the NCN grant DEC-2012/05/B/ST1/00412.}}
\author{Rafa{\l} Lata{\l}a and Marta Strzelecka}
\date{revised version}

\begin{document}

\maketitle

\begin{abstract}
We show that  for $p\geq 1$ and $r\geq 1$  the $p$-th moment of the $\ell_r$-norm of a log-concave random vector is
comparable to the sum of the first moment and the weak $p$-th moment up to a constant proportional to $r$. This extends
the previous result of Paouris concerning Euclidean norms. 
\end{abstract}

\section{Introduction and Main Results}

A measure $\mu$ on a locally convex linear space $F$ is called logarithmically concave (log-concave in short) 
if for any compact nonempty sets $K,L\subset F$ and $\lambda\in [0,1]$, 
$\mu(\lambda K+(1-\lambda)L)\geq \mu(K)^{\lambda}\mu(L)^{1-\lambda}$.
A random vector with values in $F$ is called log-concave if its distribution is logarithmically concave.
The class of log-concave measures is closed under linear transformations, convolutions and weak limits. 
By the result of Borell \cite{Bo} a $d$-dimensional vector with a full dimensional support is log-concave
iff it has a log-concave density, i.e. a density of the form $e^{-h}$, where $h$ is a convex function with values
in $(-\infty,\infty]$. A typical example of a log-concave vector is a vector uniformly distributed over a convex body.
Various results and conjectures about log-concave measures are discussed in the recently published monograph \cite{BGVV}.

One of the fundamental properties of log-concave vectors is the Paouris inequality \cite{Pa} 
(see also \cite{ALLOPT} for a shorter proof). It states that for a~log-concave vector
$X$ in $\er^n$,
\begin{equation}
\label{eq:Paour}
(\Ex\|X\|_2^{p})^{1/p}\leq 
C_1\left((\Ex \|X\|_2^2)^{1/2}+\sigma_X(p)\right) \quad \mbox{ for }p\geq 1, 
\end{equation}
where 
\[
\sigma_{X}(p):=\sup_{\|t\|_2\leq 1}\left(\Ex\left|\sum_{i=1}^nt_iX_i\right|^p\right)^{1/p}.
\]
Here and in the sequel by $C_1,C_2,\ldots$ we denote absolute constants.

It is natural to ask whether inequality \eqref{eq:Paour} may be generalized to non-Euclidean norms.
In \cite{La2} the following conjecture was formulated and discussed.

\begin{conj}
\label{conj:ws}
There exists a universal constant $C$ such that for any log-concave vector
$X$ with values in a finite dimensional normed space $(F,\|\ \|)$,  
\[
(\Ex\|X\|^p)^{1/p}\leq C\left(\Ex\|X\|+\sup_{\varphi\in F^*,\|\varphi\|_*\leq 1}(\Ex|\varphi(X)|^p)^{1/p}\right)
\quad \mbox{ for }p\geq 1.
\] 
\end{conj}

Our main result states that the conjecture holds for spaces that may be embedded in $\ell_r$ for some $r\geq 1$.

\begin{thm}
\label{thm:sublr}
Let $X$ be a log-concave vector with values in a normed space $(F,\|\ \|)$ which may be isometrically
embedded in $\ell_r$ for some $r\in [1,\infty)$. Then for $p\geq 1$,
\[
(\Ex\|X\|^p)^{1/p}\leq C_2r\left(\Ex\|X\|+\sup_{\varphi\in F^*,\|\varphi\|_*\leq 1}(\Ex|\varphi(X)|^p)^{1/p}\right).
\] 
\end{thm}

\begin{rem}
Let $X$ and $F$ be as above. Then by Chebyshev's inequality we obtain large deviation estimate for $\|X\|$:
\[
\Pr(\|X\|\geq 2eC_2rt\Ex\|X\|)\leq \exp\left(-\sigma_{X,F}^{-1}(t\Ex\|X\|)\right) \quad \mbox{ for }t\geq 1,
\]
where
\[
\sigma_{X,F}(p):=\sup_{\varphi\in F^*,\|\varphi\|_*\leq 1}(\Ex\varphi(X)^p)^{1/p}  \quad  \mbox{ for } p\geq 1
\]
denotes the weak $p$-th moment of $\|X\|$.
\end{rem}

\begin{rem}
If $i\colon F\rightarrow \ell_r$ is a nonisometric embedding and $\lambda=\|i\|_{F\to \ell_r}\|i^{-1}\|_{i(F)\to F}$,
then we may define another norm on $F$ by $\|x\|':=\|i(x)\|/\|i\|_{F\to \ell_r}$. Obviously $(F,\|\ \|')$ 
isometrically embeds in $\ell_r$, moreover $\|x\|'\leq \|x\|\leq \lambda\|x\|'$ for $x\in F$. Hence
Theorem \ref{thm:sublr} gives
\begin{align*}
(\Ex\|X\|^p)^{1/p}
&\leq \lambda (\Ex(\|X\|')^p)^{1/p}\leq
C_2r\lambda\left(\Ex\|X\|'+\sup_{\varphi\in F^*,\|\varphi\|_*^{'}\leq 1}(\Ex|\varphi(X)|^p)^{1/p}\right)
\\
&\leq C_2r\lambda\left(\Ex\|X\|+\sup_{\varphi\in F^*,\|\varphi\|_*\leq 1}(\Ex|\varphi(X)|^p)^{1/p}\right).
\end{align*}
\end{rem}

Since log-concavity is preserved under linear transformations and, by  the Hahn-Banach theorem, any linear functional on 
a subspace of $\ell_r$ is a restriction of a~functional on the whole $\ell_r$ with the same norm, it is enough to prove 
Theorem \ref{thm:sublr} for $F=\ell_r$. An easy approximation argument shows that we may consider finite dimensional 
spaces $\ell_r^n$.
To simplify the notation for an $n$-dimensional vector $X$ and $p\geq 1$ we write
\[
\sigma_{r,X}(p):=\sup_{\|t\|_{r'}\leq 1}\left(\Ex\left|\sum_{i=1}^n t_iX_i\right|^p\right)^{1/p},
\]
where $r'$ denotes the H\"older's dual of $r$, i.e. $r'=\frac{r}{r-1}$  for $r>1$ and $r'=\infty$ for $r=1$.

\begin{thm}
\label{thm:Paourlr}
Let $X$ be a finite dimensional log-concave vector and  $r\in[1,\infty)$. Then
\[
(\Ex\|X\|_r^p)^{1/p}\leq C_2r\left(\Ex\|X\|_r+\sigma_{r,X}(p)\right)\quad \mbox{ for }p\geq 1.
\] 
\end{thm}

To show the above theorem we follow the approach from \cite{La3} and establish the following result.

\begin{thm}
\label{thm:cutPaourr}
Suppose that $r\in [1,\infty)$ and $X$ is a log-concave $n$-dimensional random vector. Let 
\begin{equation}
\label{eq:defd}
d_i:=(\Ex X_i^2)^{1/2},\quad d:=\left(\sum_{i=1}^n d_i^r\right)^{1/r}.
\end{equation}
Then for $p\geq r$,
\begin{equation}
\label{eq:cutPaour}
\Ex\left(\sum_{i=1}^n|X_i|^r\ind_{\{|X_i|\geq td_i\}}\right)^{p/r}\leq (C_3r\sigma_{r,X}(p))^{p} 
\quad \mbox{ for } t\geq C_4r\log\left(\frac{d}{\sigma_{r,X}(p)}\right). 
\end{equation}
\end{thm}

\begin{rem}
Any finite dimensional space  embeds isometrically in $\ell_\infty$, so to show Conjecture~\ref{conj:ws}
it is enough to establish Theorem \ref{thm:sublr} (with a universal constant in place of $C_2r$) for $r=\infty$. 
Such  a result was shown for isotropic log-concave vectors (i.e. log-concave vectors with mean zero and identity
covariance matrix), cf.\ \cite[Corollary 3.8]{La4}. 
 However a linear image of an isotropic vector does not have to be isotropic, so to establish the conjecture
we need to consider either isotropic vectors and an arbitrary norm or vectors with a general covariance structure
and the standard $\ell_\infty$-norm.
\end{rem}

\section{Proofs}

Let us first discuss the notation. By $C$ we denote universal constants, the value of $C$ may differ
at each occurrence. Whenever we want to fix the value of an absolute constant we use letters $C_1,C_2,\ldots$.
We may always assume that $C_i\geq 1$. 
By $|I|$ we denote the cardinality of a set $I$. 
For an $n$-dimensional random vector $Z$ and $a\in \er^n$ we write $aZ$ for the vector $(a_iZ_i)_{i}$.
Observe that $\Ex\|aZ\|_2^2=\sum_{i}a_i^2\Ex Z_i^2$.

Let us recall some useful facts about log-concave vectors (for details see \cite{La3}). 
If $Z$ is log-concave real random variable then
\[
\Pr(|Z|\geq t)\leq \exp\left(2-\frac{t}{2e(\Ex Z^2)^{1/2}}\right)\quad \mbox{for } t\geq 0.
\]
Moreover, if $f:\er^n \to \er$ is a seminorm, $(\Ex f(Z)^p)^{1/p}\leq C_5\frac{p}{q}(\Ex f(Z)^q)^{1/q}$ 
for $p\geq q\geq 1$ (see \cite[Theorem 2.4.6]{BGVV}). Therefore for
any log-concave vector $X$ and any $r$,
\[
\sigma_{r,X}(\lambda p)\leq C_5\lambda \sigma_{r,X}(p)\quad \mbox{for } \lambda\geq 1,\ p\geq 2. 
\]

The Paouris inequality (\ref{eq:Paour}) together with Chebyshev's inequality imply
\begin{equation}
\label{eq:Paourtail}
\Pr\left(\|X \|_2\geq eC_1\left((\Ex \|X \|_2^2)^{1/2}+\sigma_{X}(p)\right)\right)
\leq 
e^{-p} \quad \mbox{ for }p\geq 1.
\end{equation}

The next proposition generalizes Proposition 4 from \cite{La3}.

\begin{prop}
\label{conditionalr} 
Let $X$, $r$, $d_i$, and $d$ be as in Theorem \ref{thm:cutPaourr}
and $A:=\{X\in K\}$, where $K$ is a convex set in $\er^n$ satisfying $0<\Pr(A)\leq 1/e$. Then
\\
(i) for every $t\geq r$,  
\begin{equation}
\label{eq:cond1r}
\sum_{i=1}^n\Ex |X_i|^r\ind_{A\cap\{X_i\geq td_i\}}\leq 
C_6^r\Pr(A)\left(r^r\sigma_{r,X}^r(-\log(\Pr(A)))+(dt)^re^{-t/C_7}\right).
\end{equation}
(ii) for every $t>0$, $u\geq 1$, 
\begin{align}
\notag
\sum_{k=0}^{\infty}2^{kr}\sum_{i=1}^nd_i^r&\ind_{\{\Pr(A\cap\{X_i\geq 2^k td_i\})\geq e^{-u}\Pr(A)\}}
\\
\label{eq:cond2r}
&\leq
\frac{(C_8 u)^r}{t^r}\left(\sigma_{r,X}^r(-\log(\Pr(A)))+d^r\ind_{\{t\leq uC_9\}}\right) .
\end{align}
\end{prop}

\begin{proof}
Let $Y$ be a random vector defined by 
\[
\Pr(Y\in B)=\frac{\Pr(A\cap\{X\in B\})}{\Pr(A)}=\frac{\Pr(X\in B\cap K)}{\Pr(X\in K)},  
\]
i.e. $Y$ is distributed as $X$ conditioned on $A$. Clearly, for every measurable set $B$  
one has $\Pr(X\in B)\geq \Pr(A)\Pr(Y\in B)$. It is easy to see that $Y$ is log-concave.

To simplify the notation set 
\[
p_A:=-\log \Pr(A)\quad \mbox{ and }\quad c_i:=(\Ex Y_i^2)^{1/2}, \ i=1,\ldots,n.
\] 

Let 
\[
I=I(v):=\{i\leq n\colon\ \Ex Y_i^2\geq v^2d_i^2\},
\]
where $v$ is an absolute constant to be chosen later. Let us also fix a sequence $(a_i)_{i\leq n}$.

Put $S=\sum_{i\in I}|a_{i}|c_i^{-1}Y_i^2$. Observe that $S=\|((|a_i|/c_i)^{1/2}Y_i)_{i\in I}\|_2^2$, hence by  the log-concavity
of $Y$, $\Ex S^2\leq (2C_5)^{4}(\Ex S)^2$, and the Paley-Zygmund inequality yields
\begin{equation}
\label{eq:estY1}
\Pr\left(\sum_{i\in I}|a_{i}|c_i^{-1}Y_i^2\geq \frac{1}{2}\sum_{i\in I}|a_i|c_i\right)
=\Pr\left(S\geq \frac{1}{2}\Ex S\right)\geq \frac{1}{4}\frac{(\Ex S)^2}{\Ex S^2}
\geq \frac{1}{(2\sqrt{2}C_5)^4}.
\end{equation}
We have $\Ex Y_i^4\leq (2C_5c_i)^4$, so by Chebyshev's inequality we get
\begin{equation}
\label{eq:estY2}
\Pr\left(\sum_{i\in I} |a_i|c_i^{-3}Y_i^4\geq (2C_5)^4s \sum_{i\in I} |a_i| c_i \right) \leq \frac{1}{s}
\quad \mbox{ for }s>0.
\end{equation}
Combining \eqref{eq:estY1} and \eqref{eq:estY2} we conclude that there exist constants $C_{10},C_{11}$ such that
\[
\Pr\left(\sum_{i\in I}|a_i|c_i^{-1}Y_i^2\geq \frac{1}{2} \sum_{i\in I} |a_i|c_i,\ 
\sum_{i\in I}|a_i|c_i^{-3}Y_i^4\leq C_{10}\sum_{i\in I}|a_i|c_i\right)
\geq \frac{1}{C_{11}} 
\]
and therefore
\[
\Pr\left(\sum_{i\in I}|a_i|c_i^{-1}X_i^2\geq \frac{1}{2}\sum_{i\in I}|a_i|c_i,\ 
\sum_{i\in I}|a_i|c_i^{-3}X_i^4\leq C_{10}\sum_{i\in I}|a_i|c_i\right)
\geq \frac{1}{C_{11}}\Pr(A)\geq  e^{-C_{12}p_A}. 
\]

Let $\tilde{X}$ be the vector $(|a_i|^{1/2}c_i^{-1/2}X_i)_{i\in I}$ conditioned on the set 
\[
B:=\left\{\sum_{i\in I}|a_i|c_i^{-3}X_i^4\leq C_{10}\sum_{i\in I}|a_i|c_i\right\}.
\]
Then
\begin{equation}
\label{eq:esttailtildeX}
\Pr\left(\|\tilde{X}\|_2^2\geq \frac{1}{2}\sum_{i\in I}|a_i|c_i\right)
\geq \frac{1}{\Pr(B)}e^{-C_{12}p_A}\geq e^{-C_{12}p_A}.
\end{equation}

The random vector $\tilde{X}$ is log-concave and $\Pr(B)\geq 1/2$ if $v$ is sufficiently large 
(since $\Ex X_i^4\leq Cd_i^4\leq Cv^{-4}c_i^4$ for $i\in I$). Thus
\begin{equation}
\label{eq:estvartildeXunis}
\Ex \|\tilde{X}\|_2^2=\frac{1}{\Pr(B)}\Ex\left(\sum_{i\in I} |a_i|c_i^{-1}X_i^2 \ind_B\right)
\leq 
2\sum_{i\in I}\Ex |a_i|c_{i}^{-1}d_i^2\leq 2v^{-2}\sum_{i\in I}|a_i|c_i.
\end{equation}

Now we will estimate $\sigma_{\tilde{X}}(p)$. To this end fix $t\in \er^I$ with $\|t\|_2\leq 1$.
Let $\alpha,s>0$ be numbers to be chosen later and 
\[
J_{\alpha}:=\{i\in I\colon\ |t_i|(|a_i|c_{i})^{-1/2}\leq \alpha\}. 
\]
We have
\[
\left\|\sum_{i\in J_\alpha} t_i\tilde{X}_i\right\|_p
\leq \Pr(B)^{-1/p}\left\|\sum_{i\in J_\alpha} t_i(|a_i|c_{i})^{-1/2}|a_i|X_i\right\|_p\leq
2\alpha\sigma_{1,aX}(p).
\]
Moreover
\[
\left\|\sum_{i\notin J_\alpha} t_i\tilde{X}_i { \ind}_{\{|\tilde{X}_i|\leq s(|a_i|c_i)^{1/2}\}}\right\|_p
\leq \sum_{i\notin J_\alpha} s|t_i|(|a_i|c_i)^{1/2}\leq  
s\sum_{i\notin J_\alpha} \frac{|t_i|^2}{|t_i|(|a_i|c_i)^{-1/2}}\leq \frac{s}{\alpha}\sum_{i\in I}t_i^2
\leq \frac{s}{\alpha}.
\]

Observe that by the definition of the set $B$ and the vector $\tilde{X}$ we have
\[
\sum_{i\in I}(|a_i|c_i)^{-1}\tilde{X}_i^4\leq C_{10}\sum_{i\in I}|a_i|c_i.
\]
Thus
\begin{align*}
\left\|\sum_{i\notin J_\alpha} t_i\tilde{X}_i { \ind}_{\{|\tilde{X}_i|>s(|a_i|c_i)^{1/2}\}}\right\|_p 
&\leq 
\left\|\left(\sum_{i\notin J_\alpha}\tilde{X}_i^2 { \ind}_{\{|\tilde{X}_i|>s(|a_i|c_i)^{1/2}\}}\right)^{1/2}\right\|_p 
\\ 
&
\leq \left\|\frac{1}{s}\left(\sum_{i\in I}(|a_i|c_{i})^{-1}\tilde{X}_i^4\right)^{1/2}\right\|_p 
\leq \frac{1}{s}\left(C_{10}\sum_{i\in I}|a_i|c_i\right)^{1/2}.
\end{align*}

Combining the above estimates we obtain
\[
\left\|\sum_{i\in I} t_i\tilde{X}_i\right\|_p\leq 2\alpha\sigma_{1,aX}(p)+\frac{s}{\alpha}+
\frac{1}{s}\left(C_{10}\sum_{i\in I}|a_i|c_i\right)^{1/2}.
\]
Taking the supremum over $t$  and optimizing over $\alpha>0$  we get
\begin{equation}
\label{eq:estsigmatildeXunis}
\sigma_{\tilde{X}}(p)\leq 4(s\sigma_{1,aX}(p))^{1/2}+\frac{1}{s}\left(C_{10}\sum_{i\in I}|a_i|c_i\right)^{1/2} 
\quad \mbox{ for }s>0.
\end{equation}

Paouris' inequality \eqref{eq:Paourtail} (applied to $\tilde{X}$ instead of $X$) together with 
\eqref{eq:estvartildeXunis} and
\eqref{eq:estsigmatildeXunis} implies that
\begin{align*}
\Pr\left(\|\tilde{X}\|_2\geq eC_1\left(\left(\frac{2}{v^2}\sum_{i\in I}|a_i|c_i\right)^{1/2}
+4(s\sigma_{1,aX}(C_{12}p_A))^{1/2}
+\frac{1}{s}\left(C_{10}\sum_{i\in I}|a_i|c_i\right)^{1/2}\right)\right)\phantom{a}&
\\
< e^{-C_{12}p_A}&
\end{align*}
Comparing the above with \eqref{eq:esttailtildeX} we get
\begin{align*}
eC_1\left(\left(\frac{2}{v^2}\sum_{i\in I}|a_i|c_i\right)^{1/2}+4(s\sigma_{1,aX}(C_{12}p_A))^{1/2}
+\frac{1}{s}\left(C_{10}\sum_{i\in I}|a_i|c_i\right)^{1/2}\right)\phantom{aaaaaaaaaaaa}&
\\
\geq
\left(\frac{1}{2}\sum_{i\in I}|a_i|c_i\right)^{1/2}&.
\end{align*}
If we choose $s$ and $v$ to be sufficiently large absolute constants we will get
\[
\sum_{i\in I}|a_i|(\Ex Y_i^2)^{1/2}=\sum_{i\in I}|a_i|c_i 
\leq C\sigma_{1,aX}(C_{12}p_A)\leq C\sigma_{1,aX}(p_A).
\]

Put $a_i :=( \Ex|Y_i|^2)^{(r-1)/2}\ind_{i\in I}$. 
If $\|t\|_\infty\leq 1$, then $(\sum|t_ia_i|^{r'})^{1/r'} \le \|a \|_{r'}$. Thus the previous inequality implies
\[
\sum_{i\in I} \left(\Ex |Y_i|^2\right)^{r/2}  \leq C\sigma_{1,aX}(p_A) \leq 
C\|a\|_{r'} \sigma_{r,X}(p_A) =C  \left( \sum_{i\in I} \left(\Ex |Y_i|^2\right)^{r/2} \right)^{1/r'}  \sigma_{r,X}(p_A) .
\]
This gives
\[
\sum_{i\in I} (\Ex |Y_i|^2)^{r/2}  \leq C^r  \sigma_{r,X}^r(p_A). 
\]

Since $\|Y_i\|_r\leq \max\{1,C_5r/2\}\|Y_i\|_2$ we also get
\[
\sum_{i\in I} \Ex |Y_i|^r  \leq (Cr)^r  \sigma_{r,X}^r(p_A). 
\]

To prove \eqref{eq:cond1r} note that if $i\notin I$, then $\Pr(|Y_i|\geq sd_i)\leq 2e^{-s/C}$ for $s\geq 0$, hence for $t\geq r$,
$\Ex |Y_i|^r\ind_{\{Y_i\geq td_i\}}\leq (Ctd_i)^r  e^{-t/C}$ and
\[
\sum_{i\notin I}\Ex |Y_i|^r\ind_{\{Y_i\geq t\}}\leq (Ctd)^re^{-t/C}.
\]
Hence
\begin{align*}
\frac{1}{\Pr(A)}\sum_{i=1}^n\Ex |X_i|^r\ind_{A\cap\{X_i\geq td_i\}}
&=\sum_{i=1}^n\Ex |Y_i|^r\ind_{\{Y_i\geq td_i\}}
\\
&\leq C^r\left(r^r\sigma_{r,X}^r(-\log(\Pr(A)))+(dt)^re^{-t/C}\right).
\end{align*}

To show \eqref{eq:cond2r} note first that for every $i$ the random variable $Y_i$ is log-concave, hence for $s\geq 0$,
\[
\frac{\Pr(A\cap\{X_i\geq s\})}{\Pr(A)}=
\Pr(Y_i\geq s)\leq \exp\left(2-\frac{s}{2e\|Y_i\|_2}\right).
\]
Thus, if $\Pr(A\cap\{X_i\geq 2^ktd_i\})\geq e^{-u}\Pr(A)$ and $u\geq 1$, then $\|Y_i\|_2\geq
2^ktd_i/(2e (u+2))\geq 2^ktd_i/(6eu)$.  In particular this cannot happen if $i\notin I$, $k\geq 0$ and $u\leq t/C_9$ with 
$C_9$ large enough. 
Therefore
\begin{align*}
\sum_{k=0}^{\infty}2^{kr}\sum_{i=1}^nd_i^r&\ind_{\{\Pr(A\cap\{X_i\geq 2^k td_i\})\geq e^{-u}\Pr(A)\}}
\\
&\leq
\left(\sum_{i\in I}+\ind_{\{t\leq uC_9\}}\sum_{i\notin I}\right)
d_i^r\sum_{k=0}^{\infty}2^{kr}\ind_{\{(\Ex Y_i^2)^{1/2}\geq 2^ktd_i/(6eu)\}}
\\
&\leq \left(\sum_{i\in I}+\ind_{\{t\leq uC_9\}}\sum_{i\notin I}\right)d_i^r\frac{(Cu)^r}{(td_i)^r}(\Ex Y_i^2)^{r/2}
\\
&\leq \frac{(Cu)^r}{t^r}\left(\sum_{i\in I}(\Ex Y_i^2)^{r/2}+\ind_{\{t\leq uC_9\}}\sum_{i\notin I}d_i^r\right)
\\
&\leq \frac{(Cu)^r}{t^r}\left(\sigma_{r,X}^r(-\log(\Pr(A)))+d^r\ind_{\{t\leq uC_9\}}\right). \qedhere
\end{align*}
\end{proof}

We will also use the following simple combinatorial lemma (Lemma~11 in \cite{La}).

\begin{lem}
\label{combf}
Let $l_0\geq l_1\geq\ldots\geq l_s$ be a fixed sequence of
positive integers and
\[
{\cal  F}:=\left\{f\colon\{1,2,\ldots,l_0\} \rightarrow\{0,1,2,\ldots,s\}\colon\ 
\forall_{1\leq i\leq s}\ |\{r\colon f(r)\geq i\}|\leq l_i\right\}.
\]
Then
\[
|{\cal F}|\leq\prod_{i=1}^s \left(\frac{e l_{i-1}}{l_i}\right)^{l_i}.
\]
\end{lem}

%%%%%%%%%%%%%%%%%%%%%%%%%%%%%%%%%%%%%%%%%%%%%%%%%%%%%%%%%%%%%%%%%_______________PROOF_Of_THE_MAIN_THEOREM_________

\begin{proof}[Proof of Theorem \ref{thm:cutPaourr}]
Observe that we may assume that $t\geq C_4 r$. Indeed, if $e\sigma_{r,X}(p)\leq d$ then by our assumption $t\geq C_4 r$.
If $e \sigma_{r,X}(p)> d$  then
\begin{align*}
\Bigg(\Ex\Bigg(\sum_{i=1}^n|X_i|^r&\ind_{\{|X_i|\geq td_i\}}\Bigg)^{p/r}\Bigg)^{1/p}
\\
&\leq
C_4 r \left(\sum_{i=1}^n d_i^r\right)^{1/r}
+\left(\Ex\left(\sum_{i=1}^n|X_i|^r\ind_{\{|X_i|\geq \max\{t,C_4r\}d_i\}}\right)^{p/r}\right)^{1/p}
\\
&\leq 
eC_4r\sigma_{r,X}(p)
+\left(\Ex\left(\sum_{i=1}^n|X_i|^r\ind_{\{|X_i|\geq \max\{t,C_4r\}d_i\}}\right)^{p/r}\right)^{1/p}.
\end{align*}

Moreover,  the vector $-X$ is also log-concave, has the same values of $d_i$ and $\sigma_{r,-X}=\sigma_{r,X}$. Hence
it is enough to show that
\[
\Ex\left(\sum_{i=1}^n X_i^r\ind_{\{X_i\geq td_i\}}\right)^{p/r}\leq (Cr\sigma_{r,X}(p))^{p}
\quad  \mbox{for }t\geq C_4r\max\left\{1 ,\log\left(\frac{d}{\sigma_{r,X}(p)}\right)\right\}.
\]

Observe that for $l=1,2,\ldots$,
\begin{align*}
\Ex\Bigg(\sum_{i=1}^nX_i^r \ind_{\{X_i\geq td_i\}}\Bigg)^l
& \leq 
\Ex\left(\sum_{i=1}^n\sum_{k=0}^{\infty}2^{(k+1)r}(td_i)^r\ind_{\{X_i\geq 2^ktd_i\}}\right)^l
\\ &
=(2t)^{rl}\sum_{i_1,\ldots,i_l=1}^n\sum_{k_1,\ldots,k_l=0}^{\infty} 2^{(k_1+\ldots+k_l)r}d_{i_1}^r\ldots d_{i_l}^r
\Pr(B_{i_1,k_1\ldots,i_l,k_l}),
\end{align*}
where 
\[
B_{i_1,k_1\ldots,i_l,k_l}:=\{X_{i_1}\geq 2^{k_1}td_{i_1},\ldots,X_{i_l}\geq 2^{k_l}td_{i_l}\}.
\]

Define a positive integer $l$ by 
\[
\frac{p}r< l\leq 2\frac{p}r \quad \mbox{ and }\quad l=2^{M} \mbox{ for some positive integer } M.
\]
Then $\sigma_{r,X}(p) \leq \sigma_{r,X}(rl)\leq \sigma_{r,X}(2p)\leq 2C_5\sigma_{r,X}(p)$. 
Since for any nonnegative r.v.\ $Z$ we have $(\Ex Z^{p/r})^{r/p}\leq (\Ex Z^l)^{1/l}$,  
it is enough to show that 
\begin{equation}
\label{eq:aim}
m(l)
\leq
\left(\frac{Cr\sigma_{r,X}(rl)}{t}\right)^{rl} 
\quad \mbox{ for } t\geq C_4r\max\left\{1,\log\left(\frac{d}{\sigma_{r,X}(rl)}\right)\right\},
\end{equation}
where
\[
m(l):=\sum_{k_1,\ldots,k_l=0}^{\infty}\sum_{i_1,\ldots,i_l=1}^n 2^{(k_1+\ldots+k_l)r}d_{i_1}^r\ldots d_{i_l}^r
\Pr(B_{i_1,k_1,\ldots,i_l,k_l}).
\]

We divide the sum in $m(l)$ into several parts.
Define sets 
\[
I_{0}:=\left\{(i_1,k_1,\ldots,i_l,k_l)\colon\ \Pr(B_{i_1,k_1,\ldots,i_l,k_l})> e^{-rl}\right\},
\]
and for $j=1,2,\ldots$,
\[
I_{j}:=\left\{(i_1,k_1,\ldots,i_l,k_l)\colon\ \Pr(B_{i_1,k_1,\ldots,i_l,k_l})\in (e^{-rl2^{j}},e^{-rl2^{j-1}}] \right\}.
\]
Then $m(l)=\sum_{j\geq 0}m_j(l)$, where
\[
m_{j}(l)
:= \sum_{(i_1,k_1,\ldots,i_l,k_l)\in I_j} 2^{(k_1+\ldots+k_l)r}d_{i_1}^r\ldots d_{i_l}^r\Pr(B_{i_1,k_1\ldots,i_l,k_l}).
\]

To estimate $m_{0}(l)$ define for $1\leq s\leq l$,
\[
P_sI_0:=\{(i_1,k_1,\ldots,i_s,k_s)\colon (i_1,k_1,\ldots,i_l,k_l)\in I_0 \mbox{ for some }
i_{s+1},\ldots,k_l\}.
\]
We have (since $t$ is assumed to be large)
\[
\Pr(B_{i_1,k_1,\ldots,i_s,k_s})\leq \Pr(B_{i_1,k_1})\leq \exp(2-2^{k_1-1}t/e)\leq e^{-1}.
\]
Thus for $s=1,\ldots,l-1$,  
\begin{multline*}
\sum_{(i_1, k_1,\ldots,i_{s+1}, k_{s+1})\in P_{s+1}I_{0}}
2^{(k_1+\ldots+k_{s+1})r}d_{i_1}^r\ldots d_{i_{s+1}}^r\Pr(B_{i_1,k_1,\ldots,i_{s+1},k_{s+1}})
\\
\leq
\sum_{(i_1,k_1,\ldots,i_s,k_s)\in P_{s}I_{0}}2^{(k_1+\ldots+k_{s})r}d_{i_1}^r\cdots d_{i_{s}}^r
F(i_1,k_1,\ldots,i_s,k_s),
\end{multline*}
where
\begin{align*}
F(i_1&,k_1,\ldots,i_s,k_s):=
\sum_{i=1}^n\sum_{k=0}^{\infty}
2^{kr}d_i^r\Pr(B_{i_1,k_1,\ldots,i_s,k_s}\cap\{X_i\geq 2^ktd_i\})
\\
&\leq
\sum_{i=1}^n\Ex 2t^{-r}|X_{i}|^r\ind_{B_{i_1,k_1,\ldots,i_s,k_s}\cap\{X_i\geq td_i\}}
\\
&\leq
2t^{-r}C_6^r
\Pr(B_{i_1,k_1,\ldots,i_s,k_s})\left(r^r\sigma_{r,X}^r(-\log\Pr(B_{i_1, k_1,\ldots,i_s,k_s}))+(dt)^re^{-t/C_7}\right),
\end{align*}
where the last inequality follows by \eqref{eq:cond1r}. Note that for $(i_1,k_1,\ldots,i_s,k_s)\in P_sI_0$ we have
$\Pr(B_{i_1,k_1,\ldots,i_s,k_s})> e^{-rl}$. Moreover, by our assumptions on $t$ 
(if  $C_4$ is sufficiently large with respect to $C_7$),
\[
(dt)^re^{-t/C_7}\leq t^re^{-t/(2C_7)}d^re^{-t/(2C_7)}\leq r^r\sigma_{r,X}^r(rl).
\]
Therefore
\begin{multline*}
\sum_{(i_1,k_1,\ldots,i_{s+1},k_{s+1})\in P_{s+1}I_{0}}
2^{(k_1+\ldots+k_{s+1})r}d_{i_1}^r\ldots d_{i_{s+1}}^r\Pr(B_{i_1,k_1,\ldots,i_{s+1},k_{s+1}})
\\
\leq 4t^{-r}(C_6r\sigma_{r,X}(rl))^r
\sum_{(i_1,k_1,\ldots,i_s,k_{s})\in P_{s}I_{0}}
2^{(k_1+\ldots+k_{s})r}d_{i_1}^r\ldots d_{i_s}^r\Pr(B_{i_1,k_1,\ldots,i_s,k_{s}}).
\end{multline*}

By induction we get
\begin{align*}
m_{0}(l)&
=\sum_{(i_1,k_1,\ldots,i_l,k_l)\in I_{0}}2^{(k_1+\ldots+k_{l})r}d_{i_1}^r\cdots d_{i_l}^r\Pr(B_{i_1,k_1,\ldots,i_l,k_l})
\\
&\leq
\left(\frac{{ 4}C_6r\sigma_{r,X}(rl)}{t}\right)^{r(l-1)}
\sum_{(i_1,k_1)\in P_1I_0}2^{k_1r}d_{i_1}^r\Pr(B_{i_1,k_1}).
\end{align*}
We have
\begin{align*}
\sum_{(i_1,k_1)\in P_1I_0}2^{k_1r}d_{i_1}^r\Pr(B_{i_1,k_1}) & \leq 
\sum_{i_1=1}^nd_{i_1}^r\sum_{k_1=0}^\infty 2^{k_1r}e^{2-2^{k_1-1}t/e}
\\ &
\leq \sum_{i_1=1}^n d_{i_1}^r2e^{2-t/{ (}2e{)}}\leq \left(\frac{Cr\sigma_{r,X}(rl)}{t}\right)^r,
\end{align*}
where the last two inequalities follow from the assumptions on $t$. Thus
\[
m_0(l)\leq \left(\frac{Cr\sigma_{r,X}(rl)}{t}\right)^{rl} .
\]

Now we estimate $m_j(l)$ for $j>0$. 
Fix $j>0$ and define a positive integer $\rho_1$ by 
\[
r2^{\rho_1-1}< \frac{t}{C_9}\leq r2^{\rho_1}.
\]
For all $(i_1,k_1,\ldots,i_l,k_l)\in I_j$ define a function 
$f_{i_1,k_1,\ldots,i_l,k_l}\colon \{1,\ldots,\ell\}\rightarrow \{0,1,\ldots\}$ 
by 
\[
f_{i_1,k_1,\ldots,i_l,k_l}(s):=
\left\{
\begin{array}{ll}
0 &\mbox{ if } \frac{\Pr(B_{i_1,k_1,\ldots,i_s,k_s})}{\Pr(B_{i_1,k_1,\ldots,i_{s-1},k_{s-1}})}>e^{-r},
\\
\rho &\mbox{ if } 
e^{-r2^{\rho}}< \frac{\Pr(B_{i_1,k_1,\ldots,i_s,k_s})}{\Pr(B_{i_1,k_1,\ldots,i_{s-1},k_{s-1}})}\leq e^{-r2^{\rho-1}},
\ \rho\geq 1.
\end{array}
\right.
\]
Note that for every $(i_1,k_1,\ldots,i_l,k_l) \in  I_j$ one has 
\[
1= \Pr(B_{\emptyset})\geq \Pr(B_{i_1,k_1}) \geq \Pr(B_{i_1,k_1,i_2,k_2})\geq \ldots 
\geq \Pr(B_{i_1,k_1,\ldots,i_l,k_l}) > \exp(-rl2^{j}).
\]

Denote 
\[
{\cal F}_j:=\left\{f_{i_1,k_1,\ldots,i_l,k_l}\colon\ (i_1,k_1,\ldots,i_l,k_l)\in I_j\right\}.
\]
Then for $f=f_{i_1,k_1,\ldots,i_l,k_l}\in {\cal F}_j$ and  $\rho\geq 1$ 
one has
\begin{align*}
\exp(-r2^{j}l)< \Pr(B_{i_1,k_1,\ldots,i_l,k_l}) 
&= \prod_{s=1}^{\ell}  \frac{\Pr(B_{i_1,k_1,\ldots,i_s,k_s})}{\Pr(B_{i_1,k_1,\ldots,i_{s-1},k_{s-1}})}  
\\
&\leq \exp(-r2^{\rho-1}|\{s\colon\ f(s)\geq \rho\}|).
\end{align*}
Hence for every $\rho\geq 1$ one has
\begin{equation}
\label{est_l_r}
|\{s\colon\ f(s)\geq \rho\}| \leq \min\{2^{j+1-\rho}l,l\}=:l_\rho.
\end{equation}
In particular $f$ takes values in $\{0,1,\ldots,j+1+\log_2 l\}$.
Clearly, $\sum_{\rho\geq 1}l_\rho= (j+2)l$ and $l_{\rho-1}/l_\rho\leq 2$, 
so by Lemma \ref{combf} 
\[
|{\cal F}_j| \leq \prod_{\rho=1}^{j+1+\log_2 l}\left(\frac{e l_{\rho-1}}{l_\rho} \right)^{l_\rho}\leq e^{2(j+2)l}.
\]

Now fix $f\in {\cal F}_j$ and define 
\[
I_j(f):=\{(i_1,k_1,\ldots,i_l,k_l)\colon\ f_{i_1,k_1,\ldots,i_l,k_l}=f\}
\]
and for $s\leq l$,
\[
I_{j,s}(f):=\{(i_1,k_1,\ldots,i_s,k_s)\colon\ f_{i_1,k_1,\ldots,i_l,k_l}=f\mbox{ for some }
i_{s+1},k_{s+1}\ldots,i_l,k_l\}.
\]

Recall that for $s\geq 1$, $\Pr(B_{i_1,k_1,\ldots,i_s,k_s})\leq e^{-1}$. 
Moreover for $s\leq l$,
\begin{align*}
\sigma_X(-\log\Pr(B_{i_1,k_1,\ldots,i_s,k_s}))
&\leq \sigma_X(-\log\Pr(B_{i_1,k_1,\ldots,i_l,k_l}))
\leq \sigma_X(rl2^j)
\\
&\leq C_52^{j}\sigma_X(rl). 
\end{align*}
Hence estimate \eqref{eq:cond2r} applied with $u=r2^{f(s+1)}$ implies for $1\leq s\leq l-1$,
\begin{multline*}
\sum_{(i_1,k_1,\ldots,i_{s+1},k_{s+1})\in I_{j,s+1}(f)}2^{(k_1+\ldots+k_{s+1})r}d_{i_1}^{r}\ldots d_{i_{s+1}}^{r}
\Pr(B_{i_1,k_1,\ldots,i_{s+1},k_{s+1}})
\\
\leq g(f(s+1))\sum_{(i_1,k_1,\ldots,i_{s},k_{s})\in I_{j,s}(f)}2^{(k_1+\ldots+k_{s})r} d_{i_1}^{r}\ldots d_{i_{s}}^{r} 
\Pr(B_{i_1,k_1,\ldots,i_{s},k_{s}}),
\end{multline*}
where
\[
g(\rho):=
\left\{
\begin{array}{ll}
(C_8C_5 r)^r t^{-r}2^{jr}\sigma_{r,X}(rl)^r &\mbox{ for } \rho=0,
\\
(C_8C_5r)^r t^{-r}2^{r(\rho+j)}\sigma_{r,X}(rl)^r\exp(-r2^{\rho-1})&\mbox{ for } 1\leq \rho< \rho_1,
\\
(C_8C_5r)^r t^{-r}2^{r\rho}(2^{rj}\sigma_{r, X}(rl)^r+d^r)\exp(-r2^{\rho-1})&\mbox{ for }\rho \geq \rho_1.
\end{array}
\right.
\]

Suppose that $(i_1,k_1)\in I_1(f)$ and $f(1)=\rho$. Then
\[
\exp(-r2^{\rho})\leq \Pr(X_{i_1}\geq 2^{k_1}td_{i_1})\leq \exp(2-2^{k_1-1}t/e),
\] 
hence $2^{k_1}t\leq er2^{\rho+2}$. W.l.o.g.\ $C_9>4e$, therefore $\rho\geq \rho_1$.
Moreover, $2^{rk_1}\leq (4er)^r 2^{r\rho}t^{-r}$, hence
\[
\sum_{(i_1,k_1)\in I_{j,1}(f)}2^{rk_1}d_{i_1}^{r}\Pr(B_{i_1,k_1})
\leq d^r (8er)^r t^{-r}  2^{r\rho}\exp(-r2^{\rho-1})
\leq g(\rho)=g(f(1)),
\]
since w.l.o.g. $C_8C_5\geq 8e$.
Thus an easy induction shows that
\begin{align*}
m_j(f) &:=  \sum_{(i_1,\ldots,k_{l})\in I_{j}(f)}  2^{(k_1+\ldots+k_{l})r} d_{i_1}^r\ldots d_{i_l}^r 
\Pr(B_{i_1,k_1,\ldots,i_{l},k_{l}}) 
\\
& \leq \prod_{s=1}^lg(f(s))=\prod_{\rho=0}^{\infty}g(\rho)^{n_\rho},
\end{align*}
where $n_\rho:=|f^{-1}(\rho)|$.

Observe that 
\[
e^{-r2^{j-1}l}\geq \Pr(B_{i_1,k_1,\ldots,i_l,k_l})
=\prod_{s=1}^{l}\frac{\Pr(B_{i_1,k_1,\ldots,i_s,k_s})}{\Pr(B_{i_1,k_1,\ldots,i_{s-1},k_{s-1}})}
\geq e^{-lr}\prod_{s\colon f(s)\geq 1}e^{-r2^{f(s)}}.
\]
Therefore
\[
r\sum_{\rho=1}^{\infty}n_\rho2^{\rho-1}=\frac{r}{2}\sum_{s\colon f(s)\geq 1}2^{f(s)}\geq 
\frac{r}{2}l(2^{j-1}-1).
\]
Moreover
\[
\sum_{\rho\geq 1} \rho n_\rho\leq (j+1)l+\sum_{\rho\geq j+2}\rho l_\rho= (2j+4)l.
\]
Thus
\[
\prod_{\rho=0}^{\infty}g(\rho)^{n_\rho}
\leq 
\left(\frac{C_8C_5r 2^{j}\sigma_{r,X}(rl)}{t}\right)^{rl} 2^{rl(2j+4)}
\left(1+\frac{d^r}{\sigma_{r,X}(rl)^r}\right)^m\exp\left(-\frac{rl}{2}(2^{j-1}-1)\right),
\]
where $m=\sum_{\rho\geq \rho_1}n_\rho\leq l_{\rho_1}\leq 2^{j+1-\rho_1}l$.
By the assumption on $t$ we have $1+d^r/\sigma_{r,X}(rl)^r\leq 2\exp(t/C_4)\leq \exp(r2^{\rho_1-4})$ if $C_4$ 
is large enough (with respect to $C_9$).
Hence 
\[
m_j(l)\leq |{\cal F}_j| \left(\frac{\sqrt{e}C_8C_5 2^{(3j+4)}r\sigma_{r,X}(rl)}{t}\right)^{rl}\exp(-rl2^{j-3}).
\]

We get
\[
m(l)= \sum_{j=0}^{\infty} m_j(l)\leq  \left(\frac{Cr\sigma_{r,X}(rl)}{t}\right)^{rl}+
\sum_{j= 1}^\infty\left(\frac{C2^{5j}r\sigma_{r,X}(rl)}{t}\right)^{rl}\exp(-rl2^{j-3}).
\]
To finish the proof of \eqref{eq:aim}, note that 
\begin{align*}
& \sum_{j=1}^\infty \left(2^{5j}\right)^{rl}\exp(-rl2^{j-3}) \leq C^{rl} \sum_{j=1}^\infty \exp(-rl2^{j-4}) 
\leq C^{rl}. \qedhere
\end{align*} 
\end{proof}

\begin{proof}[Proof of Theorem \ref{thm:Paourlr}]
Since $(\Ex\|X\|_r^p)^{1/p}\leq C_5p\Ex\|X\|_r$, we may assume that
$p\geq r$. Let $d_i$ and $d$ be as in Theorem \ref{thm:cutPaourr}. Then

\begin{equation*}
d=\| (\Ex X_i^2)^{1/2} \|_r \le { 2C_5 \| (\Ex |X_i|) \|_r\le 2C_5\Ex\|X\|_r}.
\end{equation*}
Set
\[
\tilde{p}:=\inf\{q\geq p\colon\ \sigma_{r,X}(q)\geq d\}. 
\]
Theorem \ref{thm:cutPaourr} applied with $\tilde{p}$ instead of $p$ and $t=0$ yields
\begin{align*}
(\Ex\|X\|_r^p)^{1/p}&\leq (\Ex\|X\|_r^{\tilde{p}})^{1/\tilde{p}}\leq C_3r\sigma_{r,X}(\tilde{p})
=C_3r\max\{d,\sigma_{r,X}(p)\}
\\
&\leq Cr(\Ex\|X\|_r+\sigma_{r,X}(p)). \qedhere
\end{align*}
\end{proof}

\end{document}